\documentclass[preprint,10pt]{elsarticle}
\usepackage{amssymb}
\usepackage{amsthm}
\usepackage{amsmath,amscd,amsfonts,bm}
\usepackage{setspace}
\usepackage{epsfig}
\usepackage{graphicx}
\usepackage[english]{babel}
\usepackage{fancyhdr}

\DeclareMathOperator{\diag}{diag}
\newcommand\R{\mathbb R}
\theoremstyle{plain}
\newtheorem{thm}{Theorem}

\newtheorem{lem}[thm]{Lemma}

\newtheorem{Def}{Definition}
\theoremstyle{remark}

\journal{Applied Mathematics Letters}

\begin{document}

\begin{frontmatter}
\title{Observability for Initial Value Problems with Sparse Initial Data}
\author{Nicolae Tarfulea}
\ead{tarfulea@calumet.purdue.edu}
\ead[url]{http://ems.calumet.purdue.edu/tarfulea}
\address{Department of Mathematics, Purdue University Calumet, Hammond, IN 46323}
\begin{abstract}
In this work we introduce the concept of $s$-sparse observability for large systems of ordinary
differential equations. Let $\dot x=f(t,x)$ be such a system. At time $T>0$, suppose 
we make a set of observations $b=Ax(T)$ of the solution of the system with initial data $x(0)=x^0$,
where $A$ is a matrix satisfying the restricted isometry property. 
The aim of this paper is to give answers to the following questions: 
Given the observation $b$, is $x^0$ uniquely determined knowing that
$x^0$ is sufficiently sparse? Is there any way to reconstruct such a sparse initial data $x^0$?
\end{abstract}


\end{frontmatter}

\section{Introduction and Results}
In recent years a number of papers on signal processing have developed a series of 
ideas and techniques on the reconstruction of a finite signal $x\in \R^m$ from many fewer observations 
than traditionally believed necessary. It is now common knowledge that it is possible to exactly recover 
$x$ knowing that it is sparse or nearly sparse in the sense that it has only a limited number of nonzero components.
A more formal definition of sparsity can be given through the $l_0$ norm $\| x\|_0:=\#\{ i:\, x_i\neq 0\}$, that is,
the cardinality of $x$'s support. If $\| x\|_0\leq s$, for $s$ a nonnegative integer, then we say that $x$ is $s$-sparse.

Since sparsity is a very often encountered feature in signal processing and many other 
mathematical models of real-life phenomena, estimation under 
sparsity assumption has been a topic of increasing interest in the last decades. At this point, the
work on this subject is so extended and growing so rapidly that it is extremely difficult to mention without
injustice its achievements and results. It is beyond the scope of this paper to review, even partially, the
contributions to this new and very dynamic area of research. For example, in the signal processing case, 
the interested reader can find valuable insight in the very informative survey by Bruckstein, Donoho and Elad \cite{BDE}.

This work addresses the recovery of the initial state of a high-dimensional dynamic variable from a restricted set
of measurements, knowing that the initial state is sparse. More precisely, let $x(\cdot)$ be the solution 
of the following initial-value problem
\begin{equation}\label{diffeq}
\dot x(t)=f(t,x(t)),\mbox{ for } t>0;\quad x(0)=x^0.
\end{equation}

Suppose that we can observe 
\begin{equation}
\label{observations}
b=Ax(T)
\end{equation}
at a certain time $T>0$, where the vector $b$ represents the observations, and 
$A$ is an $n\times m$ measurement matrix (dictionary).  As in signal processing case, 
the more interesting situation is when $n<<m$; one interprets $b$ as low-dimensional 
observations/measurements at time $T>0$ of the high-dimensional dynamic solution $x(\cdot)$. 
Here are two interesting questions that we address in this note: 

{\bf Question 1:} Given the observation $b$, is $x^0$ uniquely determined knowing that
$x^0$ is sufficiently sparse? 

{\bf Question 2:} Is there any way to reconstruct such a sparse initial data $x^0$?

Hereafter, $f:[0,T]\times\R^m\to\R^m$ is a Lipschitz function in the second variable, i.e.,
there is $L(\cdot)\in L^\infty (\mathbf R^m)$ such that $\| f(t,x)-f(t,y)\|\leq L(t)\| x-y\|$
in $[0,T]\times\R^m$. By $L$ we denote the $L^\infty$-norm of $L(\cdot)$ over the interval $[0,T]$.
For $x\in\mathbf R^m$, the $l^p$-norm ($p\geq 1$) of $x$ is defined as usually $\| x\|_p:=(\sum_{i=1}^m|x_i|^p)^{1/p}$.
In what follows, we also assume that the matrix $A$ satisfies the restricted isometry property. 
Let us recall the concept of restricted isometry constants (see \cite{CT}).
\begin{Def}
For each integer $s=1,2,\ldots$ define the isometry constant $\delta_s$ of a matrix $A$ as the smallest number such that
$(1-\delta_s)\| x\|_2^2\leq\| Ax\|_2^2\leq (1+\delta_s)\| x\|_2^2$ 
holds for all $x\in \mathbf R^m$ with $\| x\|_0\leq s$. 
\end{Def}
The following definition introduces a new concept, that is, the notion of $s$-sparse observability. 
\begin{Def}
The pair \eqref{diffeq}-\eqref{observations} is called $s$-sparse observable at time $T>0$ if the
knowledge of $b$ allows us to compute the $s$-sparse initial data vector $x^0$.
\end{Def}
In other words, \eqref{diffeq}-\eqref{observations} is $s$-sparse observable at time $T>0$ if for
any solutions $x_1(\cdot)$ and $x_2(\cdot)$ corresponding to the $s$-sparse intial data $x_1^0$ and $x_2^0$, 
respectively, with $Ax_1(T)=Ax_2(T)$, we have $x_1^0=x_2^0$.

Our first result gives a positive answer to Question 1. 
In fact it provides a sufficient condition for $s$-sparse observability. 
Roughly speaking, it says that $s$-sparse observability holds for sufficiently short periods of time.
\begin{thm}
\label{observable}
Let $T<\frac{1}{L}\ln{(1+\frac{\sqrt{1-\delta_{2s}}}{\| A\|})}$. Then 
\eqref{diffeq}-\eqref{observations} is $s$-sparse observable at time $T$. 
\end{thm}
In the remainder of this section we indicate how the sparsest initial data $x^0$ can be found, or approximated. 
We consider the more applicable situation in which the measurements at time $T$ are corrupted with noise. That is,
\begin{equation}
\label{observations_corrupted}
b=Ax(T)+e,
\end{equation}
where $e$ is the noise term whose maximum magnitude is $\epsilon$ (i.e., $\| e\|_2\leq \epsilon$).

Suppose we seek the sparsest initial data $x^0$ that solves \eqref{diffeq} and \eqref{observations_corrupted}. 
In order to narrow down to one well-defined (sparse) solution, we 
consider the problem: $$\mbox{{\bf $(P_0)$}}\quad\mbox{ Find }\hat{x}:=\arg\min_{\| b-Ax(T)\|_2\leq\epsilon}\| x\|_0.$$
Here $x(t)$ is the solution of \eqref{diffeq} together with the initial condition $x(0)=x$. 
However, this is a very hard combinatorial-dynamic optimization problem and practically impossible to solve. 
We propose reconstructing $x^0$ by 
replacing the $l_0$ norm $\| \cdot\|_0$ with the weighted $l_1$ norm $\|\cdot\|_{1,w}$
($\| x\|_{1,w}:=\sum_{i=1}^mw_i|x_i|$, $w_i>0$, $i=1,2,...,m$), which
is, in a natural way, its best convex approximant. This strategy originates in the work of Santosa and Symes
\cite{SS} in the mid-eighties (see also \cite{DL,DS} for early results). In this context, we seek $x^0$ as the
solution to the dynamic optimization problem
$$\mbox{{\bf $(P_{1,w}^\epsilon)$}}\quad\mbox{ Find }x^*:=\arg\min_{\| b-Ax(T)\|_2\leq\epsilon}\| x\|_{1,w}.$$ 
Denote by $\tau$ the condition number of the matrix $W:=\diag\{w_i\}$, that is, 
$\tau:=\max_{1\leq i\leq m}\{w_i\}/\min_{1\leq i\leq m}\{w_i\}$, 
and by $x_s$ the vector having only the
$s$ largest entries of the vector $x$, the others being set to zero. 
The following result gives a positive answer to Question 2. In essence, it states that for sufficiently small times,
the accurate recovery of the sparse initial data $x^0$ can be done. Its proof is largely influenced 
by the methods and techniques used in \cite{C,CRT}. 
\begin{thm} 
\label{recover}
If $\delta_{2s}<(1+\tau\sqrt{2})^{-1}$ and 
$T<\frac{1}{L}\ln{\Big( 1+\frac{1-\delta_{2s}(1+\tau\sqrt{2})}{(1+\tau )\| A\|\sqrt{1+\delta_{2s}}}\Big )}$, 
then the solution $x^*$ to $(P_{1,w}^\epsilon)$ 
satisfies 
$\| x^*-x^0\|_2\leq C_0s^{-1/2}\| x^0-x_s^0\|_1+C_1\epsilon$, with $C_0$ and $C_1$ constants independent of $x^0$.
\end{thm}
\section{Proofs of Results}
\subsection{Proof of Theorem~\ref{observable}}
Suppose that \eqref{diffeq}-\eqref{observations} is not $s$-sparse observable at time $T$. That is, there
exist $s$-sparse vectors $x_1^0$ and $x_2^0$, with $x_1^0\neq x_2^0$, such that $Ax_1(T)=Ax_2(T)$, where
$x_1(t)$ and $x_2(t)$ are solutions to \eqref{diffeq} together with the initial data $x_1(0)=x_1^0$ and
$x_2(0)=x_2^0$, respectively. 
Since $x_1(t)=x_1^0+\int_0^tf(s,x_1(s))ds$ and $x_2(t)=x_2^0+\int_0^tf(s,x_2(s))ds$, it follows that 
$\| x_2(t)-x_1(t)\|_2 \leq \| x_2^0-x_1^0\|_2+\int_0^t\| f(s,x_2(s))-f(s,x_1(s))\|_2ds
\leq \|x_2^0-x_1^0 \|_2+\int_0^tL(s)\| x_2(s)-x_1(s)\|_2ds $. 
By Gronwall's inequality, one obtains 
$\| x_2(t)-x_1(t)\|_2 \leq \| x_2^0-x_1^0\|_2e^{\int_0^tL(s)ds},$ 
and so
\begin{align}
\nonumber
0 =\| Ax_2(T)-Ax_1(T)\|_2 &=\| A (x_2^0-x_1^0)+A\int_0^T[f(t,x_2(t))-f(t,x_1(t))]dt\|_2 \\ \nonumber
&\geq \| A(x_2^0-x_1^0)\|_2-\| A\| \int_0^TL(t)\| x_2(t)-x_1(t)\|_2dt \\ \nonumber
&\geq \| A(x_2^0-x_1^0)\|_2-\| A\| \int_0^TL(t)\| x_2^0-x_1^0\|_2e^{\int_0^tL(s)ds}dt\\ \nonumber
&\geq \| A(x_2^0-x_1^0)\|_2-M\| A\|\cdot\| x_2^0-x_1^0\|_2,
\end{align}
where $M:=e^{\int_0^TL(s)ds}-1$. Thus, $\| A(x_2^0-x_1^0)\|_2\leq M\| A\|\cdot\| x_2^0-x_1^0\|_2$.
Then, because $x_2^0-x_1^0$ is $2s$-sparse and from the restricted isometry property, we obtain that
$\sqrt{1-\delta_{2s}}\leq M\| A\|$,
which cannot hold for $T<\frac{1}{L}\ln{(1+\frac{\sqrt{1-\delta_{2s}}}{\| A\|})}$.

\subsection{Proof of Theorem~\ref{recover}}
The following Lemma is a simple application of the parallelogram identity and is due to Cand\'{e}s 
\cite[Lemma 2.1.]{C}. We include it here, together with its proof, for reader's convenience.
\begin{lem}
\label{candes}
Let $x$ and $x'$ be two vectors in $\mathbf R^m$. Suppose that $x$ and $x'$ are supported on disjoint subsets
and $\| x\|_0\leq s$ and $\| x'\|_0\leq s'$. Then $|\langle Ax,Ax'\rangle |\leq\delta_{s+s'}\| x\|_2\| x'\|_2.$
\end{lem}
\begin{proof}
Denote by $y=x/\| x\|_2$ and $y'=x'/\| x'\|_2$ the unit vectors in the $x$ and $x'$ directions, respectively. 
By the restricted isometry property, it is easy to see that 
$2(1-\delta_{s+s'})\leq \| Ay\pm Ay'\|_2^2\leq 2(1+\delta_{s+s'}).$ 
These inequalities, together with the parallelogram identity, give 
$|\langle Ay,Ay'\rangle |\ =\frac{1}{4}| \| Ay+ Ay'\|_2^2-\| Ay- Ay'\|_2^2|\leq \delta_{s+s'}$, 
and so $|\langle Ax,Ax'\rangle |\leq\delta_{s+s'}\| x\|_2\| x'\|_2$,  
which concludes the proof.
\end{proof}
Let $y(t)$ be the solution to the initial value problem $\dot y(t)=f(t,y)$ in $(0,T)$, $y(0)=x^*$, 
where $x^*$ is the solution to $(P_{1,w}^\epsilon)$. Then,
\begin{equation}
\label{ineq0}
\| Ay(T)-Ax(T)\|_2\leq \| b-Ay(T)\|_2+\| b-Ax(T)\|_2\leq 2\epsilon.
\end{equation}
Decompose $x^*$ as $x^*=x^0+h$. Then, as in \cite{C}, decompose $h$ into a sum of vectors $h_{T_0}$,
$h_{T_1}$, ..., each of $s$-sparsity. Here, $T_0$ corresponds to the locations of the first $s$ largest 
coefficients of $h$, $T_1$ to the locations of the next $s$ largest coefficients, and so on.
Since $x(t)=x^0+\int_0^tf(s,x(s))ds$ and $y(t)=x^*+\int_0^tf(s,y(s))ds$, it follows that 
$\| y(t)-x(t)\|_2 \leq \| h\|_2+\int_0^tL(s)\| y(s)-x(s)\|ds$. 
By Gronwall's inequality, we get
$\| y(t)-x(t)\|_2 \leq \| h\|_2e^{\int_0^tL(s)ds},\mbox{ for all }t\in [0,T].$ 
Then,
\begin{align}
\nonumber
\| Ay(T)-Ax(T)\|_2 &=\| Ah+A\int_0^T[f(t,y(t))-f(t,x(t))]dt\|_2 \\ \nonumber
&\geq \| Ah\|_2-\| A\| \int_0^TL(t)\| y(t)-x(t)\|_2dt \\ \nonumber
&\geq \| Ah\|_2-\| A\| \int_0^TL(t)\| h\|_2e^{\int_0^tL(s)ds}dt\\ \label{ineq5}
&\geq \| Ah\|_2-M\| A\|\| h\|_2,
\end{align}
where $M:=e^{\int_0^TL(s)ds}-1$. From \eqref{ineq0} and \eqref{ineq5}, we obtain
\begin{equation}
\label{iineq1}
\| Ah\|_2\leq 2\epsilon+M\| A\| \|h\|_2.
\end{equation}
Next, we estimate $\| h_{(T_0\cup T_1)^c}\|_2$, where $h_{(T_0\cup T_1)^c}:=h-h_{T_0}-h_{T_1}$. 
First off, observe that
$\| h_{T_j}\|_2\leq s^{1/2}\| h_{T_j}\|_{\infty}\leq s^{-1/2}\| h_{T_{j-1}}\|_1$, $j\geq 2$, 
and so
\begin{equation}
\label{ineq3} 
\sum_{j\geq 2}\| h_{T_j}\|_2\leq s^{-1/2}\sum_{j\geq 1}\| h_{T_j}\|_1=s^{-1/2}\| h_{T_0^c}\|_1,
\end{equation}
where $h_{T_0^c}:=h-h_{T_0}$. Thus,
\begin{equation}
\label{ineq1}
\| h_{(T_0\cup T_1)^c}\|_2=\Big\| \sum_{j\geq 2}h_{T_j}\Big\|_2\leq\sum_{j\geq 2}\| h_{T_j}\|_2
\leq s^{-1/2}\| h_{T_0^c}\|_1.
\end{equation}
Because $x^0$ is feasible, it satisfies $\| x^*\|_{1,w}\leq \| x^0\|_{1,w}$, which implies 
\begin{multline*}
\sum_{i=1}^mw_i|x_i^0|\geq \sum_{i=1}^mw_i|x_i^0+h_i|=\sum_{i\in T_0}w_i|x_i^0+h_i|+\sum_{i\in T_0^c}w_i|x_i^0+h_i|\\
\geq \sum_{i\in T_0}w_i|x_i^0|-\sum_{i\in T_0}w_i|h_i|-\sum_{i\in T_0^c}w_i|x_i^0|+\sum_{i\in T_0^c}w_i|h_i|,
\end{multline*}
and so 
$\| h_{T_0^c}\|_{1,w}\leq\| h_{T_0}\|_{1,w}+2\| x^0-x^0_s\|_{1,w}$. 
This last inequality induces 
$\min_{1\leq i\leq m}\{w_i\}\| h_{T_0^c}\|_1\leq\max_{1\leq i\leq m}\{w_i\}(\|h_{T_0}\|_1+2\| x^0-x^0_s\|_1)$, 
and so 
\begin{equation}
\label{ineq2}
\| h_{T_0^c}\|_1\leq\tau (\| h_{T_0}\|_1+2\| x^0-x^0_s\|_1).
\end{equation}
From \eqref{ineq1}, \eqref{ineq2}, and the Cauchy-Schwarz inequality, it follows that
\begin{equation}
\label{fivestars}
\| h_{(T_0\cup T_1)^c}\|_2 \leq\tau s^{-1/2}(\| h_{T_0}\|_1+2\| x^0-x^0_s\|_1\leq\tau (\| h_{T_0}\|_2+2e_0),
\end{equation}
with $e_0:=s^{-1/2}\| x^0-x^0_s\|_1$.
Now, let us estimate $\| h_{T_0\cup T_1}\|_2$. By the restricted isometry property, it follows that
\begin{equation}
\label{i1}
\frac{1}{\sqrt{1+\delta_{2s}}}\| Ah_{T_0\cup T_1}\|_2\leq\| h_{T_0\cup T_1}\|_2\leq \frac{1}{\sqrt{1-\delta_{2s}}}
\| Ah_{T_0\cup T_1}\|_2.
\end{equation}
Since $Ah_{T_0\cup T_1}=Ah-\sum_{j\geq 2}Ah_{T_j}$, we have that 
$\| Ah_{T_0\cup T_1}\|_2^2 =\langle Ah_{T_0\cup T_1}, Ah\rangle-\sum_{j\geq 2}\langle Ah_{T_0\cup T_1},Ah_{T_j}\rangle
\leq |\langle Ah_{T_0\cup T_1}, Ah\rangle |+\sum_{j\geq 2}|\langle Ah_{T_0\cup T_1},Ah_{T_j}\rangle|$. 
This and the restricted isometry property imply that  
\begin{align}
\nonumber
|\langle Ah_{T_0\cup T_1}, Ah\rangle | &\leq \| Ah_{T_0\cup T_1}\|_2\| Ah\|_2
\leq \| Ah_{T_0\cup T_1}\|_2 (2\epsilon +M\| A\| \| h\|_2)\quad \mbox{(by \eqref{iineq1})}\\ \label{i3}
&\leq \sqrt{1+\delta_{2s}}\| h_{T_0\cup T_1}\|_2(2\epsilon +M\| A\| \| h\|_2) \quad \mbox{(by \eqref{i1})}.
\end{align}
From Lemma~\ref{candes}, we have that 
$|\langle Ah_{T_i}, Ah_{T_j}\rangle |\leq \delta_{2s}\| h_{T_i}\|_2\| h_{T_j}\|_2$, for $i\neq j$, and so 
$|\langle Ah_{T_0\cup T_1}, Ah_{T_j}\rangle |\leq \delta_{2s}(\| h_{T_0}\|_2+\| h_{T_1}\|_2)\| h_{T_j}\|_2
\leq \sqrt{2} \delta_{2s}\| h_{T_0\cup T_1}\|_2\| h_{T_j}\|_2$. 
Therefore,
\begin{equation}
\label{iineq2}
\| Ah_{T_0\cup T_1}\|_2^2\leq\sqrt{1+\delta_{2s}}\| h_{T_0\cup T_1}\|_2(2\epsilon +M\| A\| \| h\|_2)+
\sqrt{2}\delta_{2s}\| h_{T_0\cup T_1}\|_2\sum_{j\geq 2}\| h_{T_j}\|_2.
\end{equation}
As in \cite{C}, let $\alpha =2\sqrt{1+\delta_{2s}}(1-\delta_{2s})^{-1}$ and 
$\rho=\sqrt{2}\delta_{2s}(1-\delta_{2s})^{-1}$. Then, 
\begin{align}
\nonumber
\| h_{T_0\cup T_1}\|_2 &\leq \alpha (\epsilon +\frac{1}{2}M\| A\|\| h\|_2)+\rho \sum_{j\geq 2}\| h_{T_j}\|_2
\quad \mbox{(by \eqref{i1} and \eqref{iineq2})}\\ \nonumber
&\leq \alpha (\epsilon +\frac{1}{2}M\| A\|\| h\|_2)+\rho s^{-1/2}\| h_{T_0^c}\|_1
\quad \mbox{(by \eqref{ineq3})}\\ \nonumber
&\leq \alpha (\epsilon +\frac{1}{2}M\| A\|\| h\|_2)+\rho s^{-1/2}\tau (\| h_{T_0}\|_1+2\| x^0-x^0_s\|_1)
\quad \mbox{(by \eqref{ineq2})}\\ \nonumber
&\leq \alpha (\epsilon +\frac{1}{2}M\| A\|\| h\|_2)+\rho \tau \| h_{T_0\cup T_1}\|_2+2\rho\tau e_0,
\end{align}
and so
\begin{equation}
\label{i6}
\| h_{T_0\cup T_1}\|_2 \leq (1-\rho\tau)^{-1}[\alpha (\epsilon +\frac{1}{2}M\| A\|\| h\|_2)+2\rho\tau e_0].
\end{equation}
(Observe that $1-\rho\tau>0$ since $\delta_{2s}<(1+\sqrt{2}\tau)^{-1}$.) 

Then, 
\begin{align*}
\| h\|_2 &\leq \| h_{T_0\cup T_1}\|_2+\| h_{(T_0\cup T_1)^c}\|_2
\leq (1+\tau)\| h_{T_0\cup T_1}\|_2+2\tau e_0 \quad \mbox{(by \eqref{fivestars})}\\
&\leq (1+\tau)(1-\rho\tau)^{-1}[\alpha (\epsilon +\frac{1}{2}M\| A\|\| h\|_2)+2\rho\tau e_0]+2\tau e_0, 
\quad \mbox{(by \eqref{i6})}
\end{align*}
and so $\| h\|_2 \leq C_0e_0+C_1\epsilon $, with $C_0:=2\tau (\rho +1)[1-\rho\tau-0.5\alpha (1+\tau)M\| A\|]^{-1}$
and $C_1:=\alpha (1+\tau)[1-\rho\tau-0.5\alpha (1+\tau)M\| A\|]^{-1}$. As a consequence of 
$T<\frac{1}{L}\ln{\Big( 1+\frac{1-\delta_{2s}(1+\tau\sqrt{2})}{(1+\tau )\| A\|\sqrt{1+\delta_{2s}}}\Big )}$, 
observe that both $C_0$ and $C_1$ are strictly positive.

\bigskip

{\bf Acknowledgments.} This work was supported by the 2009 Northwest Indiana Computational Grid (NWICG) Summer Program.

\end{document}